\documentclass{amsart}

\usepackage{dsfont}
\usepackage[centertags]{amsmath}
\usepackage[usenames,dvipsnames]{xcolor}
\usepackage[colorlinks=true,linkcolor=Red,urlcolor=NavyBlue,citecolor=blue]{hyperref}
\usepackage{amsfonts}
\usepackage{MnSymbol}
\usepackage{amsthm}
\usepackage{newlfont}
\usepackage{amscd}
\usepackage{amsmath}
\usepackage{enumerate}
\usepackage[all,2cell]{xy}
\usepackage{tikz-cd}
\usetikzlibrary{positioning, calc}
\usepackage{mathrsfs}
\UseAllTwocells
\input xy
\xyoption{2cell}
\xyoption{all}
\usepackage{verbatim}
\usepackage{eucal}
\usepackage{graphicx}

\usepackage{pict2e}

\makeatletter
\newcommand{\adjunction}{\@ifstar\named@adjunction\normal@adjunction}
\newcommand{\normal@adjunction}[4]{%
	#1\colon #2%
	\mathrel{\vcenter{%
			\offinterlineskip\m@th
			\ialign{%
				\hfil$##$\hfil\cr
				\longrightharpoonup\cr
				\noalign{\kern-.3ex}
				\smallbot\cr
				\longleftharpoondown\cr
			}%
	}}%
	#3 \noloc #4%
}
\newcommand{\named@adjunction}[4]{%
	#2%
	\mathrel{\vcenter{%
			\offinterlineskip\m@th
			\ialign{%
				\hfil$##$\hfil\cr
				\scriptstyle#1\cr
				\noalign{\kern.1ex}
				\longrightharpoonup\cr
				\noalign{\kern-.3ex}
				\smallbot\cr
				\longleftharpoondown\cr
				\scriptstyle#4\cr
			}%
	}}%
	#3%
}
\newcommand{\longrightharpoonup}{\relbar\joinrel\rightharpoonup}
\newcommand{\longleftharpoondown}{\leftharpoondown\joinrel\relbar}
\newcommand\noloc{%
	\nobreak
	\mspace{6mu plus 1mu}
	{:}
	\nonscript\mkern-\thinmuskip
	\mathpunct{}
	\mspace{2mu}
}
\newcommand{\smallbot}{%
	\begingroup\setlength\unitlength{.15em}%
	\begin{picture}(1,1)
		\roundcap
		\polyline(0,0)(1,0)
		\polyline(0.5,0)(0.5,1)
	\end{picture}%
	\endgroup
}
\makeatother

\makeatletter
\newcommand*{\doublerightarrow}[2]{\mathrel{
		\settowidth{\@tempdima}{$\scriptstyle#1$}
		\settowidth{\@tempdimb}{$\scriptstyle#2$}
		\ifdim\@tempdimb>\@tempdima \@tempdima=\@tempdimb\fi
		\mathop{\vcenter{
				\offinterlineskip\ialign{\hbox to\dimexpr\@tempdima+1em{##}\cr
					\rightarrowfill\cr\noalign{\kern.5ex}
					\rightarrowfill\cr}}}\limits^{\!#1}_{\!#2}}}
\newcommand*{\triplerightarrow}[1]{\mathrel{
		\settowidth{\@tempdima}{$\scriptstyle#1$}
		\mathop{\vcenter{
				\offinterlineskip\ialign{\hbox to\dimexpr\@tempdima+1em{##}\cr
					\rightarrowfill\cr\noalign{\kern.5ex}
					\rightarrowfill\cr\noalign{\kern.5ex}
					\rightarrowfill\cr}}}\limits^{\!#1}}}
\makeatother

\newcommand{\LL}{\mathbb{L}}

\def\A{\mathcal{A}}
\def\B{\mathcal{B}}
\def\QQ{\mathbb{Q}}
\def\C{\mathcal{C}}
\def\D{\mathcal{D}}

\def\SS{\mathcal{S}}

\def\I{\mathcal{I}}
\def\cS{\mathcal{S}}

\def\P{\mathcal{P}}
\def\Q{\mathcal{Q}}

\def\U{\mathcal{U}}
\def\V{\mathcal{V}}

\def\O{\mathcal{O}}

\def\lrarsimeq{\overset{\simeq}{\lrar}}

\newcommand{\HSwarrow}{\kern0.05ex\vcenter{\hbox{\Huge\ensuremath{\Swarrow}}}\kern0.05ex}
\newcommand{\hSwarrow}{\kern0.05ex\vcenter{\hbox{\huge\ensuremath{\Swarrow}}}\kern0.05ex}
\newcommand{\LLSwarrow}{\kern0.05ex\vcenter{\hbox{\LARGE\ensuremath{\Swarrow}}}\kern0.05ex}
\newcommand{\LSwarrow}{\kern0.05ex\vcenter{\hbox{\Large\ensuremath{\Swarrow}}}\kern0.05ex}

\newcommand{\HSearrow}{\kern0.05ex\vcenter{\hbox{\Huge\ensuremath{\Searrow}}}\kern0.05ex}
\newcommand{\hSearrow}{\kern0.05ex\vcenter{\hbox{\huge\ensuremath{\Searrow}}}\kern0.05ex}
\newcommand{\LLSearrow}{\kern0.05ex\vcenter{\hbox{\LARGE\ensuremath{\Searrow}}}\kern0.05ex}
\newcommand{\LSearrow}{\kern0.05ex\vcenter{\hbox{\Large\ensuremath{\Searrow}}}\kern0.05ex}

\newcommand{\HDownarrow}{\kern0.05ex\vcenter{\hbox{\Huge\ensuremath{\Downarrow}}}\kern0.05ex}
\newcommand{\hDownarrow}{\kern0.05ex\vcenter{\hbox{\huge\ensuremath{\Downarrow}}}\kern0.05ex}
\newcommand{\LLDownarrow}{\kern0.05ex\vcenter{\hbox{\LARGE\ensuremath{\Downarrow}}}\kern0.05ex}
\newcommand{\LDownarrow}{\kern0.05ex\vcenter{\hbox{\Large\ensuremath{\Downarrow}}}\kern0.05ex}

\newcommand{\HUparrow}{\kern0.05ex\vcenter{\hbox{\Huge\ensuremath{\Uparrow}}}\kern0.05ex}
\newcommand{\hUparrow}{\kern0.05ex\vcenter{\hbox{\huge\ensuremath{\Uparrow}}}\kern0.05ex}
\newcommand{\LLUparrow}{\kern0.05ex\vcenter{\hbox{\LARGE\ensuremath{\Uparrow}}}\kern0.05ex}
\newcommand{\LUparrow}{\kern0.05ex\vcenter{\hbox{\Large\ensuremath{\Uparrow}}}\kern0.05ex}

\newcommand\restr[2]{{
		\left.\kern-\nulldelimiterspace 
		#1 
		\vphantom{\big|} 
		\right|_{#2} 
}}

\DeclareMathOperator{\Id}{Id}

\DeclareFontFamily{OT1}{pzc}{}
\DeclareFontShape{OT1}{pzc}{m}{it}{<-> s * [1.10] pzcmi7t}{}
\DeclareMathAlphabet{\mathpzc}{OT1}{pzc}{m}{it}

\DeclareMathOperator{\Alg}{Alg}

\DeclareMathOperator{\Free}{Free}

\DeclareMathOperator{\op}{op}

\DeclareMathOperator{\Map}{Map}

\DeclareMathOperator{\Sets}{Sets}

\DeclareMathOperator{\Cat}{Cat}

\DeclareMathOperator{\Fun}{Fun}

\DeclareMathOperator{\Ob}{Ob}

\DeclareMathOperator{\Ho}{Ho}

\DeclareMathOperator{\Op}{Op}

\DeclareMathOperator{\Tw}{Tw}

\DeclareMathOperator{\BMod}{BMod}

\DeclareMathOperator{\Aut}{Aut}

\DeclareMathOperator{\h}{h}

\DeclareMathOperator{\sMod}{sMod}

\DeclareMathOperator{\defi}{def}

\DeclareMathOperator{\Hom}{Hom}
\DeclareMathOperator{\rL}{L}

\DeclareMathOperator{\sN}{N}

\DeclareMathOperator{\Coll}{Coll}

\DeclareMathOperator{\Set}{Set}

\DeclareMathOperator{\E}{E}

\DeclareMathOperator{\End}{End}
\DeclareMathOperator{\Fin}{Fin}

\DeclareMathOperator{\cor}{core}

\DeclareMathOperator{\sB}{b}

\def\x{\overset}

\def\Hom{\textrm{Hom}}
\def\End{\textrm{End}}

\newcommand{\tgpd}{\kern0.05ex\vcenter{\hbox{\footnotesize\ensuremath{2}}}\kern0.05ex\mathcal{G}pd} 

\def\lrar{\longrightarrow}

\def\ovl{\overline}

\newtheorem{theorem}{Theorem}[section]
\newtheorem{lemma}[theorem]{Lemma}

\theoremstyle{definition}
\newtheorem{definition}[theorem]{Definition}
\newtheorem{example}[theorem]{Example}

\newtheorem{construction}[theorem]{Construction}

\theoremstyle{definition}
\newtheorem{remark}[theorem]{Remark}

\theoremstyle{definition}
\newtheorem{notation}[theorem]{Notation}

\theoremstyle{definition}

\theoremstyle{corollary}
\newtheorem{corollary}[theorem]{Corollary}

\theoremstyle{proposition}
\newtheorem{proposition}[theorem]{Proposition}

\theoremstyle{observation}

\numberwithin{equation}{section}

\theoremstyle{definition}
\newtheorem{conv}[theorem]{Convention}

\begin{document}

\title{Mapping spaces between operads in relation to bimodules}

\author{Truong Hoang}
\address{DEPARTMENT OF MATHEMATICS, HANOI FPT UNIVERSITY, VIET NAM.}

\email{truonghm@fe.edu.vn}



\begin{abstract} 
	 This paper investigates mapping spaces between enriched operads and relates these spaces to those between operadic bimodules via convenient fiber sequences. The main statements hold for simplicial operads, operads enriched in simplicial modules over a commutative ring, and for connective dg operads, with the last case relying on an operadic version of the Dold-Kan correspondence.
\end{abstract}

\maketitle

\tableofcontents

\section{Introduction}\label{s:intro}

In abstract homotopy theory, \textbf{(derived) mapping spaces} serve as a fundamental tool, encoding higher homotopical information of maps between objects. Such spaces are typically rigorously defined and studied within the $\infty$-categorical setting, or any framework suitable for homotopy theory, such as model (or semi-model) categories. In this paper, we study mapping spaces between operads enriched over a \textbf{symmetric monoidal model category} (see \cite{Hovey}). Specifically, we establish fiber sequences that relate the mapping spaces of operads to those of their associated bimodules.

\smallskip

Let $(\cS, \otimes, 1_{\cS})$ be a symmetric monoidal model category, which we refer to as the \textbf{base category}. For a set $C$, we let $\Op_{C}(\cS)$ denote the \textit{category of $C$-colored operads in $\cS$}, endowed with the  model structure induced from that on $\cS$ (see $\S$\ref{s:22}). We begin by investigating the mapping spaces between these objects, building upon the insights from Dwyer-Hess \cite{Hess}.

\smallskip

Suppose we are given a map $f : \P\lrar\O$ in $\Op_C(\mathcal{S})$. We denote by $\BMod(\P)$ the \textit{category of $\P$-bimodules}, equipped again with the model structure transferred from $\cS$. We also regard $\O$ as a $\P$-bimodule with the structure induced by $f$.  

\begin{theorem}\label{t:main1}(\ref{t:goodness}) Suppose that $\P$ is a cofibrant operad satisfying the stability hypothesis presented in Definition \ref{S8}. Then there is a fiber sequence of spaces
	\begin{equation}\label{eq:goodnessmain2}
		\Omega\Map^{\h}_{\Op_C(\mathcal{S})}(\P, \O) \lrar \Map^{\h}_{\BMod(\P)}(\P, \O) \lrar \underset{c\in C}{\prod} \Map^{\h}_{\cS}(1_{\cS}, \O(c;c))
	\end{equation}
in which the loop space $\Omega\Map^{\h}_{\Op_C(\mathcal{S})}(\P, \O)$ is taken at the base point $f$.
\end{theorem}

\begin{remark} This theorem extends a result from \cite{Hess}, originally established in the context of \textbf{non-symmetric simplicial operads}. Furthermore, for the case of (symmetric) simplicial operads, in which $\O$ satisfies $\O(c;c)\simeq *$ for every color $c$, then \eqref{eq:goodnessmain2} yields a weak equivalence
	$$ \Omega\Map^{\h}_{\Op_C(\mathcal{S})}(\P, \O) \lrarsimeq \Map^{\h}_{\BMod(\P)}(\P, \O),$$
as observed by Ducoulombier \cite{Julien}.
\end{remark}

Next we denote by $\Op(\cS)$ the \textit{category of $\cS$-enriched operads} (with non-fixed sets of colors), endowed with the \textbf{Dwyer-Kan  model structure} (see $\S$\ref{s:22}). A relation between the mapping spaces between enriched operads and those between operadic bimodules is achieved by combining \eqref{eq:goodnessmain2} with the statement below. 

\smallskip

Let $\Q\in\Op_{D}(\mathcal{S})$ be another operad colored in some set $D$. Suppose we are given a map $g : C \lrar D$ between the sets of colors. We will write $g^*\Q$ for the $C$-colored operad induced by the restriction of colors, and write $\Q_1$ for the \textit{underlying category} of $\Q$ (see $\S$\ref{s:21}). 
\begin{theorem}\label{t:simplifying}(\ref{p:simplifying})  There is a fiber sequence of spaces
	\begin{equation}\label{eq:simplifying1}
		\Map^{\h}_{\Op_C(\cS)}(\P,g^*\Q) \lrar \Map^{\h}_{\Op(\cS)}(\P,\Q) \lrar \underset{C}{\prod}\cor(\Q_1)
	\end{equation}
	in which $\cor(\Q_1)$ refers to the \textit{core of} $\Q_1$ (see Notation \ref{no:core}).
\end{theorem}

For an illustration, when $\C$ is a (fibrant) simplicial category, we have $\cor(\C) \simeq \sN(\C)^\simeq$, ie, the maximal $\infty$-groupoid contained in the  nerve of $\C$.

\smallskip

Combining the two theorems above, we obtain a significant result as follows. In the following, we will assume that $\cS$ satisfies additional conditions, particularly the \textbf{invertibility hypothesis} introduced by Lurie. Suppose we are given a map $f : \P \lrar \O$ in $\Op(\cS)$. 

\begin{theorem}\label{t:double1}(\ref{t:double}) Suppose that $\P$ is a cofibrant operad satisfying the stability hypothesis. Then there is a natural weak equivalence 
	$$  \Omega^2\Map^{\h}_{\Op(\cS)}(\P,\O) \simeq \Omega\Map^{\h}_{\BMod(\P)}(\P,f^*\O) $$
	in which the double loop space on the left is taken at the base point $f$, and while the loop space on the right is taken at the induced map $\P\lrar f^*\O$.
\end{theorem}

\begin{remark} The base categories of primary interest in this paper are \textit{simplicial sets}, \textit{simplicial $\mathbf{k}$-modules}, and \textit{non-negatively dg $R$-modules}, where $\textbf{k}$ and $R$ are commutative rings, with $R$ containing the field $\QQ$ (see Example \ref{ex:basis}). When $\cS$ is any of these categories, Theorems \ref{t:main1} and \ref{t:double1} remain valid for every $\Sigma$-cofibrant operad $\P$.
\end{remark}

\section{Background and notations}\label{s:backgr}

\smallskip

\subsection{Operads and modules over an operad}\label{s:21}

Let $(\cS, \otimes, 1_\cS)$ be a symmetric monoidal category. For a set $C$, regarded as the \textbf{set of colors}, we denote by $\Op_C(\cS)$ the \textbf{category of (symmetric) $C$-colored operads in $\cS$}. For more information, $\Op_C(\cS)$ agrees with the category of monoids in the monoidal category $(\Coll_C(\cS), \circ, \I_C)$ in which

\smallskip

$\bullet$ an object of $\Coll_C(\cS)$, referred to as a $C$-\textbf{collection} (in $\cS$), is a collection $$M=\{M(c_1,\cdots,c_n;c) \, | \, c_i,c\in C, n\geq 0\}$$ of objects in $\cS$ equipped with a right $\Sigma_n$-action whose data consist of maps of the form 
$$\sigma^{*} : M(c_1,\cdots,c_n;c) \lrar M(c_{\sigma(1)},\cdots,c_{\sigma(n)};c)$$
for every sequence $(c_1,\cdots,c_n;c)$ and every permutation $\sigma\in \Sigma_n$,

\smallskip

$\bullet$ $-\circ-$ is the well known \textbf{composite product} of $C$-collections, and

\smallskip

$\bullet$ the monoidal unit $\I_C$ is given by $\I_C(c;c)=1_{\cS}$ for every $c\in C$, and  $\emptyset_\cS$ (ie, the initial object of $\cS$) otherwise. (See, for example, \cite{Pavlov} for more details on these definitions.)

\smallskip

Let $\P\in\Op_C(\cS)$ be a $C$-colored operad in $\cS$.  By considering $\P$ as a monoid object, it is natural to define the notions of \textbf{left (right) modules} and \textbf{bimodules} over $\P$, as described in the definition below.

\begin{definition} \label{d:operadicmodules}
	
	\begin{enumerate}[(1)]
		
		\item A \textbf{left $\mathcal{P}$-module} is a $C$-collection $M$  endowed with an action map $\mathcal{P}\circ M\lrar M$ whose data consist of $\Sigma_*$-equivariant maps of the form
		$$  \P(c_1,\cdots,c_n;c) \otimes M(d_{1,1},\cdots,d_{1,k_1};c_1) \otimes \cdots \otimes M(d_{n,1},\cdots,d_{n,k_n};c_n)$$
		$$ \lrar M(d_{1,1},\cdots,d_{1,k_1},\cdots,d_{n,1},\cdots,d_{n,k_n};c).$$
		These are required to satisfy the essential associativity and unitality axioms for left modules.
		
		\item Dually, a \textbf{right $\mathcal{P}$-module} is a $C$-collection $M$ that is equipped with an action map $M \circ \P \lrar M$ whose data consist of $\Sigma_*$-equivariant maps of the form
		$$  M(c_1,\cdots,c_n;c) \otimes \P(d_{1,1},\cdots,d_{1,k_1};c_1) \otimes \cdots \otimes \P(d_{n,1},\cdots,d_{n,k_n};c_n) $$
		$$\lrar M(d_{1,1},\cdots,d_{1,k_1},\cdots,d_{n,1},\cdots,d_{n,k_n};c) $$
		satisfying the associativity and unitality axioms for right modules.
		
		\item A $\P$\textbf{-bimodule} is a $C$-collection $M$ structured as both a left and a right $\P$-module, with these structures satisfying a compatibility requirement.
		
	\end{enumerate}
\end{definition}

Moreover, a \textbf{$\P$-algebra} is by definition a left $\P$-module $A$ with $A(c_1,\cdots,c_n;c)=\emptyset_\cS$ whenever $n\neq 0$. We will denote by $\BMod(\P)$ the \textit{category of $\P$-bimodules}, and by $\Alg_\P(\cS)$ the \textit{category of $\P$-algebras}.

\smallskip

We recall the following construction from \cite{Hoang}, which establishes an operad encoding the category of $\P$-bimodules. (See also \cite[$\S$2.1.1]{Julien1}.)

\begin{construction}\label{con:BP} We denote by $\Fin$ the \textit{skeleton of the category of finite sets}, where the objects are written as $\underline{0} := \emptyset$ and $\underline{m} := \{1,\cdots,m\}$ for $m\geq1$. We construct an $\cS$-enriched operad, denoted $\textbf{B}^{\mathcal{P}}$, as follows. The set of colors is given by $C$-sequences $\{(c_1,\cdots,c_n;c) \, | \, c_i,c\in C, n\geq 0\}$. A typical space of $n$-ary operations is given by
	$$ \textbf{B}^{\mathcal{P}}\left( \,    \left(c_1,\cdots,c_{r_1}; c^{(1)} \right) ,    \left(c_{r_1+1},\cdots,c_{r_1+r_2}; c^{(2)} \right) , \cdots ,  \left(c_{r_1+\cdots+r_{n-1}+1},\cdots,c_{r_1+\cdots+r_{n}}; c^{(n)} \right)  ;   \left(d_1,\cdots,d_m;d \right) \, \right) $$
	$$ := \bigsqcup_{\underline{m} \x{f}{\lrar} \underline{r_1+\cdots+r_{n}}} \left [ \mathcal{P}\left (c^{(1)},\cdots,c^{(n)};d \right ) \otimes \bigotimes_{i=1,\cdots,r_1+\cdots+r_{n}} \mathcal{P} \left (\{d_j\}_{j\in f^{-1}(i)};c_i \right ) \right ] $$
	where the coproduct ranges over the hom-set $\Hom_{\Fin}(\underline{m} , \underline{r_1+\cdots+r_{n}})$. The operad structure of $\textbf{B}^{\mathcal{P}}$ is canonically defined via the structure of $\P$.  
\end{construction}

The main interest in this construction is as follows.

\begin{proposition}\label{p:BP} \textup{(\cite[$\S$2.2]{Hoang})} There is a categorical isomorphism 
	$$ \Alg_{\textbf{B}^\P}(\cS) \cong \BMod(\P)$$
between $\textbf{B}^\P$-algebras and $\P$-bimodules.
\end{proposition}

One can organize all the operads enriched in $\cS$, with arbitrary sets of colors, into a single category as follows. We associate to each map $\alpha : C \lrar D$ of sets a  functor $\alpha^{*} : \Op_D(\SS) \lrar \Op_C(\SS)$ given by taking $\Q\in\Op_D(\SS)$ to $\alpha^{*}\Q$ with 
$$ \alpha^{*}\Q(c_1,\cdots,c_n;c) := \Q(\alpha(c_1),\cdots,\alpha(c_n);\alpha(c)) .$$
\begin{definition} The \textbf{category of $\cS$-enriched operads}, denoted $\Op(\cS)$, is defined to be the \textit{contravariant Grothendieck construction}
	$$ \Op(\SS) := \int_{C\in \Sets} \Op_C(\SS).$$
	More precisely, an object of $\Op(\SS)$ is a pair $(C,\P)$ with $C\in \Sets$ and $\P\in \Op_C(\SS)$, and such that a morphism $(C,\P) \lrar (D,\Q)$ consists of a map $\alpha : C \lrar D$ of sets and a map $f :  \P \lrar \alpha^{*}\Q$ of $C$-colored operads. 
\end{definition}

\begin{remark} For a $C$-colored operad $\P$ in $\cS$, there exists an  adjunction between under categories:
\begin{equation}\label{eq:signi}
	\adjunction*{}{\Op_C(\cS)_{\P/}}{\Op(\cS)_{\P/}}{}
\end{equation}
in which the left adjoint is given by the obvious embedding, and while the right adjoint is induced by the restriction of colors.
\end{remark}

\begin{remark}\label{r:signi1} We will write $\Cat(\cS)$ for the \textbf{category of $\cS$-enriched categories}. Clearly each $\cS$-enriched category can be identified with an $\cS$-enriched operad concentrated in arity $1$. Conversely, each object $\P \in \Op(\cS)$ admits an \textbf{underlying category}, denoted $\P_1 \in \Cat(\cS)$, and formed by the collection of $1$-ary operations of $\P$. Moreover, these two  constructions determine an adjunction
\begin{equation}\label{eq:signi1}
	\adjunction*{}{\Cat(\cS)}{\Op(\cS)}{}.
\end{equation}
\end{remark}

\smallskip

\subsection{Operadic model structures}\label{s:22}

We now recall the homotopy theory of algebras over an operad and the homotopy theory of enriched operads themselves.

\smallskip

Let us assume further that $\cS$ is a \textbf{symmetric monoidal model category}. By convention, we will say that $\cS$ is \textbf{admissible} if for every set $C$ and for every operad $\P \in \Op_{C}(\cS)$, the category of $\P$-algebras inherits a model structure transferred from the product model structure on $\cS^{\times C}$ along the free-forgetful adjunction
	$$ \adjunction*{}{\cS^{\times C}}{\Alg_\P(\cS)}{}$$
(see eg, \cite{Pavlov, Ieke2}). Moreover, $\P\in \Op_{C}(\cS)$ is $\Sigma$-\textbf{cofibrant} if it is cofibrant as a $C$-collection with respect to the projective model structure.

\begin{remark}\label{r:encompass} Note that for any set $C$, there exists an operad in $\cS$, denoted $S^C$, such that there is an equivalence of categories $\Alg_{S^C}(\cS) \simeq \Op_C(\cS)$ (cf \cite[$\S$3]{Guti}). Hence, the above admissibility condition already encompasses the transferred model structure on the category $\Op_C(\cS)$. The same observation applies to operadic bimodules, due to Proposition \ref{p:BP}. 
\end{remark}

The \textbf{homotopy category} of an $\cS$-enriched category $\C$, denoted $\Ho(\C)$, is the ordinary category with the same objects as $\C$ and morphisms defined by  $$\Hom_{\Ho(\C)}(x,y) := \Hom_{\Ho(\cS)}(1_\cS,\Map_\C(x,y)) .$$
By convention, the \textbf{homotopy category of an operad} $\P \in \Op(\cS)$ is $$\Ho(\P) := \Ho(\P_1)$$
(see Remark \ref{r:signi1} for notation). Moreover, a map $f : \P \lrar \Q$ in $\Op(\cS)$ is called a \textit{levelwise weak equivalence} (resp. \textit{fibration, trivial fibration}, etc) if for every sequence $(c_1,\cdots,c_n;c)$ of colors of $\P$, the induced map $$\P(c_1,\cdots,c_n;c) \lrar \Q(f(c_1),\cdots,f(c_n) ; f(c))$$ is a weak equivalence (resp. fibration, trivial fibration, etc) in $\cS$. In what follows, we recall the two model structures on enriched operads, according to Caviglia's \cite{Caviglia}.

\begin{definition}\label{d:DKop} A map $f : \P \lrar \Q$ in $\Op(\cS)$ is called a \textbf{Dwyer-Kan equivalence} if it is a levelwise weak equivalence and such that the induced functor $$\Ho(f):\Ho(\P)\lrar\Ho(\Q)$$ between homotopy categories is essentially surjective.
\end{definition}

\begin{definition}\label{d:DKope} The \textbf{Dwyer-Kan model structure on} $\Op(\cS)$ is the one whose weak equivalences are the Dwyer-Kan equivalences and whose trivial fibrations are the levelwise trivial fibrations surjective on colors. 
\end{definition}

\begin{definition}\label{d:canoope} The \textbf{canonical model structure on} $\Op(\cS)$ is the one whose fibrant objects are the levelwise fibrant operads and whose trivial fibrations are the same as those of the Dwyer-Kan model structure. 
\end{definition}

We will use the term \textit{canonical equivalence} (resp. \textit{fibration}) to refer to a weak equivalence (resp. fibration) with respect to the canonical model structure.

\begin{remark}\label{rem:catop} The two model structures on $\Op(\cS)$ naturally generalize the corresponding model structures on $\Cat(\cS)$, so that the embedding $\Cat(\cS) \lrar \Op(\cS)$ creates all three classes of structure maps. For more information, the Dwyer-Kan model structure on $\Cat(\cS)$ is developed in \cite{Luriehtt, Muro}, and can be thought of as a generalization of the Dwyer-Kan homotopy theory of simplicial categories introduced by Bergner \cite{Bergner}. On the other hand, the canonical model structure on $\Cat(\cS)$ is established by Berger-Moerdijk \cite{Ieke}, as another approach to the homotopy theory of enriched categories.  
\end{remark}

\begin{remark}\label{rem:canfib} A map $f$ in $\Op(\cS)$ is a Dwyer-Kan equivalence if and only if it is a levelwise weak equivalence and such that its underlying map in $\Cat(\cS)$ is a Dwyer-Kan equivalence. On other hand, as originally introduced in \cite{Caviglia}, $f$ is a canonical equivalence (resp. fibration) if and only if it is a levelwise weak equivalence (resp. fibration) and its underlying map in $\Cat(\cS)$ is a canonical equivalence (resp. fibration).
\end{remark}

\begin{remark}\label{rem:coincide} In practice, the Dwyer-Kan and canonical model structures usually coincide. For example, this holds true as soon as $\cS$ satisfies the assumptions of \cite[Proposition 2.20]{Ieke}. 
\end{remark}

\begin{conv}\label{conv:adsu} We will say that $\cS$ is \textbf{sufficient} if it is admissible, and such that the two model structures on $\Op(\cS)$ (as well as on $\Cat(\cS)$) exist and coincide. In this situation, we will use the term \textit{Dwyer-Kan} to refer to the common model structure.
\end{conv}

\begin{example}\label{ex:basis} We discuss the three base categories that we are concerned with in this paper. 
	
	\smallskip
	
(i) The first one is $\Set_\Delta$ the category of \textit{simplicial sets}, equipped with the Cartesian monoidal structure and with the classical (Kan-Quillen) model structure. The category $\Set_\Delta$ is sufficient.
	
	\smallskip

(ii) For a commutative ring $\textbf{k}$, we let $\sMod(\textbf{k})$ denote the category of \textit{simplicial $\textbf{k}$-modules}. This category is equipped with the degreewise tensor product over $\textbf{k}$ and with a model structure transferred from the classical model structure on $\Set_\Delta$. As in the case above, the category $\sMod(\textbf{k})$ is sufficient.

\smallskip

(iii) Let $R$ be a commutative ring containing the field $\QQ$ of rational numbers. We also regard the category $\C_{\geqslant0}(R)$ of \textit{connective dg $R$-modules} (ie, the differential graded $R$-modules concentrated in non-negative degrees). It comes equipped with the usual tensor product; along with the \textit{projective model structure}, whose weak equivalences are precisely the quasi-isomorphisms and whose fibrations are given by the chain maps that are surjective at every positive degree. One can show that the category $\C_{\geqslant0}(R)$ is sufficient as well.

\smallskip

 For more details, we refer the reader to \cite{Pavlov, Caviglia}.
\end{example}

\smallskip

\subsection{Operadic Dold-Kan correspondence}\label{s:23}

We recall a version of the Dold-Kan correspondence for operadic algebras and describe some noteworthy corollaries derived from it. Let $R$ be a commutative ring containing the field $\QQ$. As usual, we let $$ \sN :  \sMod_{R} \lrar  \C_{\geqslant0}(R)$$ denote the \textbf{normalization functor}. It is known that $\sN$ determines a categorical equivalence (cf \cite{Dold}). However, this no longer holds when extended to the context of operadic algebras (see \cite{Schwede}); instead, it can still be realized as a Quillen equivalence, as presented below.  

\smallskip

Recall that the \textit{shuffle maps} $\nabla : \sN A \otimes \sN B \lrarsimeq \sN(A \otimes B)$ for $A, B \in \sMod_{R}$, which are weak equivalences, establish a \textit{lax symmetric monoidal} structure on $\sN$ (see, eg, \cite{Schwede, Weibel, Goerss} for details). Now, let $C$ be a fixed set of colors, and $\P$ a $C$-colored operad in $\sMod_{R}$. Applying the normalization functor levelwise to $\P$ produces a $C$-colored operad in $\C_{\geqslant0}(R)$, denoted $\sN\P$. Moreover, applying the normalization functor to each $\P$-algebra $A$ yields an $\sN\P$-algebra, written as $\sN_\P(A) := \{\sN A(c)\}_{c\in C}$. 

\smallskip

The following theorem follows as a specific instance of \cite[Theorem 4.3.2]{WY}.
\begin{theorem}\label{p:main} \textup{(White-Yau)} Suppose that $\P\in\Op_C(\sMod_{R})$ is $\Sigma$-cofibrant.  Then the functor $$\sN_\P : \Alg_\P(\sMod_{R}) \lrar \Alg_{\sN\P}(\C_{\geqslant0}(R))$$  is the right adjoint of a Quillen equivalence.
\end{theorem}  

\begin{remark} If $R$ is any commutative ring and the $\Sigma$-cofibrancy condition on $\P$ holds, the categories $\Alg_\P(\sMod_{R})$ and $\Alg_{\sN\P}(\C_{\geqslant0}(R))$ are equipped with a transferred semi-model structure (see, eg, \cite{Fresse1, Spitzweck}). In this case, the statement above remains valid when interpreted within the framework of semi-model categories. The same holds for Corollaries \ref{co:dkbimod} and \ref{ex:tmain} below. 
\end{remark}

\begin{corollary}\label{co:dkbimod} Suppose that $\P\in\Op_C(\sMod_{R})$ is $\Sigma$-cofibrant. Then the functor 
	$$ \sN^{\sB}_\P : \BMod(\P) \lrar \BMod(\sN\P) $$
	defined by applying the normalization functor levelwise is the right adjoint of a Quillen equivalence.
	\begin{proof} According to Proposition \ref{p:BP},  there is a categorical isomorphism $$\Alg_{\textbf{B}^\P}(\sMod_{R}) \cong \BMod(\P).$$ 
	Additionally, $\textbf{B}^\P$ is $\Sigma$-cofibrant due to that property of $\P$. Applying Theorem \ref{p:main} to $\textbf{B}^\P$ gives us a right Quillen equivalence:
		$$ \Alg_{\textbf{B}^\P}(\sMod_{R}) \lrarsimeq \Alg_{\sN\textbf{B}^\P}(\C_{\geqslant0}(R)).$$ 
		On other hand, the shuffle map $\nabla$ induces a map $\nabla^{*} : \textbf{B}^{\sN\P} \lrar \sN\textbf{B}^\P$ of operads in $\C_{\geqslant0}(R)$. Clearly $\nabla^{*}$ is a weak equivalence between $\Sigma$-cofibrant operads. Due to \cite[Theorem 12.5.A]{Fresse1}, we obtain another right  Quillen equivalence:
		$$ \Alg_{\sN\textbf{B}^\P}(\C_{\geqslant0}(R)) \lrarsimeq \Alg_{\textbf{B}^{\sN\P}}(\C_{\geqslant0}(R)).$$
		We complete the proof by observing that the composition of the two right Quillen equivalences obtained above coincides with the functor $\sN^{\sB}_\P$.	
	\end{proof}
\end{corollary}

\begin{corollary}\label{ex:tmain} The normalization functor induces a right Quillen equivalence
	\begin{equation}\label{eq:opnormalized1}
		\sN_C : \Op_C(\sMod_{R}) \lrarsimeq \Op_C(\C_{\geqslant0}(R))
	\end{equation}
	between $C$-colored operads.	
	\begin{proof} Consider the operad $S^C$ that encodes the category of $C$-colored operads in some symmetric monoidal category (see Remark \ref{r:encompass}). Note that $S^C$ comes from an operad in $\Sets$, and is equipped with a free action by the symmetric groups. In particular, $S^C$ is $\Sigma$-cofibrant when considered in any base category as long as the monoidal unit is cofibrant. Applying Theorem \ref{p:main} to this operad, we obtain a right Quillen equivalence:
		$$ \sN_{S^C} : \Alg_{S^C}(\sMod_{R}) \lrarsimeq \Alg_{\sN S^C}(\C_{\geqslant0}(R)). $$
		Now, since $S^C$ is discrete, it implies that $\sN S^C$ coincides with $S^C$ itself considered as an operad in $\C_{\geqslant0}(R)$. So the two functors $\sN_C$ and $\sN_{S^C}$ are isomorphic. The proof is therefore completed.
	\end{proof}	
\end{corollary}

\begin{remark}\label{r:opbivar} We now describe a variant of Theorem \ref{p:main}. Let $\rL_C$ denote the left adjoint to the functor $\sN_C$. For a cofibrant operad $\O \in\Op_C(\C_{\geqslant0}(R))$, the unit map $\eta_\O : \O \lrar \sN_C\rL_C(\O)$ is a weak equivalence (between $\Sigma$-cofibrant operads) by the Quillen equivalence $\rL_C \dashv \sN_C$ and the fact that the normalization functor creates weak equivalences. We hence obtain a composed Quillen equivalence $$\adjunction*{\simeq}{\Alg_\O(\C_{\geqslant0}(R))}{\Alg_{\sN_C\rL_C(\O)}(\C_{\geqslant0}(R))}{} \adjunction*{\simeq}{}{\Alg_{\rL_C\O}(\sMod_{R})}{} $$
in which the first adjunction is induced by $\eta_\O$ and the second is given by applying Theorem \ref{p:main} to the operad $\rL_C\O \in  \Op_C(\sMod_{R})$. The same argument can be applied to obtain a Quillen equivalence of the form
\begin{equation}\label{eq:opbivar}
	\adjunction*{\simeq}{\BMod(\O)}{\BMod(\rL_C\O)}{}
\end{equation}
in which the right adjoint is given by applying the normalization functor levelwise.
\end{remark}

\section{Main statements}\label{s:mapping}

In this section, we will let $\cS$ be a symmetric monoidal category and let $\P$ be a $C$-colored operad in $\cS$ with $C$ being a fixed set of colors.  To avoid confusion, we will write
  $$ \sqcup \; , \; \bigsqcup \; \text{and} \; \coprod $$
to denote the coproduct operations (and also, pushouts) in the categories $\BMod(\mathcal{P})$, $\Op_C(\mathcal{S})$ and $\Op(\mathcal{S})$, respectively.

\subsection{The first fiber sequence}\label{s:mappingC}

First, let us consider the case where the set of colors is fixed. 

\smallskip

We are recalling a hypothesis introduced by Dwyer-Hess, that has played a very active role in their work \cite{Hess}. Recall by definition that a \textbf{pointed $\P$-bimodule} is a $\P$-bimodule $M$ endowed with a map $\I_C \lrar M$ of $C$-collections. We will let $\BMod(\mathcal{P})^{*}$ denote the category of such objects with morphisms being the basepoint-preserving maps.

\begin{remark} Note that $\mathcal{P}$ itself is an object of $\BMod(\mathcal{P})^{*}$ with the base point given by the unit of $\P$. Notice also that $\BMod(\mathcal{P})^{*}$ agrees with the under-category $\BMod(\mathcal{P})_{\mathcal{P}\circ \mathcal{P}/}$ where $\mathcal{P}\circ \mathcal{P}$ represents the free $\P$-bimodule generated by $\mathcal{I}_C$.
\end{remark}
Let $f + g : \mathcal{P}\bigsqcup\mathcal{P} \lrar \Q$ be a map in $\Op_C(\mathcal{S})$. Then $\Q$ inherits a $\P$-bimodule structure with the left (resp. right) $\P$-action induced by $f$ (resp. $g$).  From this, we obtain a restriction functor $$\Op_C(\mathcal{S})_{\mathcal{P}\bigsqcup\mathcal{P}/} \lrar \BMod(\mathcal{P})^{*} ,$$ which admits a left adjoint, denoted by $\E$.  Observe that the functor $\E$ sends $\P$ to itself $\P\in \Op_C(\mathcal{S})_{\mathcal{P}\bigsqcup\mathcal{P}/}$ equipped with the fold map $\Id+\Id : \mathcal{P}\bigsqcup\mathcal{P} \lrar \P$. 

\smallskip

We further suppose that the base category $\cS$ is admissible as defined in $\S$\ref{s:22}.

\begin{definition}\label{S8} (\textbf{Stability hypothesis})
	The operad $\P\in \Op_C(\mathcal{S})$ is said to be \textbf{intrinsically stable} if the left derived  functor  $\LL\E:\BMod(\mathcal{P})^{*}\lrar \Op_C(\mathcal{S})_{\mathcal{P}\bigsqcup\mathcal{P}/}$ sends $\P$ to itself $\mathcal{P}$. Moreover, we say that the base category $\cS$ is \textbf{adequate} if every cofibrant operad in $\cS$ is intrinsically stable.
\end{definition} 

As we will see below, the base categories $\Set_\Delta$, $\sMod_{\textbf{k}}$ and $\C_{\geqslant0}(R)$ (cf Example \ref{ex:basis}) are all adequate, for any commutative ring $\textbf{k}$ and any commutative ring $R$ containing the field $\QQ$.

\begin{remark}\label{r:goodness} Considering operads as monoid objects in the monoidal category of $C$-collections, we can see that the \textit{stability hypothesis} is a particular case of what Dwyer and Hess refer to as the stability of \textit{distinguished objects} with respect to the functor $\E$ (cf \cite[$\S$3.4]{Hess}).
\end{remark}

We give the following statement as an illustration of how that hypothesis can be applied in practice.  

\begin{theorem}\label{t:goodness} Suppose that the base category $\cS$ is admissible, that $\P\in \Op_C(\mathcal{S})$ is cofibrant and intrinsically stable, and suppose given a map $f : \P \lrar \O$ in $\Op_C(\mathcal{S})$. Then there is a fiber sequence of spaces
	\begin{equation}\label{eq:goodnessmain}
		\Omega\Map^{\h}_{\Op_C(\mathcal{S})}(\P, \O) \lrar \Map^{\h}_{\BMod(\P)}(\P, \O) \lrar \underset{c\in C}{\prod}\Map^{\h}_{\cS}(1_{\cS}, \O(c;c))
	\end{equation}
in which the loop space $\Omega\Map^{\h}_{\Op_C(\mathcal{S})}(\P, \O)$ is taken at the base point $f$, and $\O$ is regarded as a $\P$-bimodule with the structure induced by $f$.
\end{theorem}

\begin{remark} This theorem is in fact a version of \cite[Theorem 3.11]{Hess}. Nonetheless, we shall provide an alternative proof that is somewhat more direct.
\end{remark}

\begin{proof}[\underline{Proof of Theorem~\ref{t:goodness}}] First, observe that there is a commutative square of left Quillen functors
	\begin{equation}\label{eq:goodness}
		 \xymatrix{
			\BMod(\mathcal{P})_{\mathcal{P}\circ \mathcal{P}/} \ar^-{\mu_!}[r]\ar_-{\E}[d] & \BMod(\mathcal{P})_{\P/} \ar^-{\ovl{\E}}[d] \\
			\Op_C(\mathcal{S})_{\mathcal{P}\bigsqcup\mathcal{P}/} \ar_-{(\Id+\Id)_!}[r] & \Op_C(\mathcal{S})_{\P/} \\
		}	
	\end{equation}
where $\ovl{\E}$ denotes the left adjoint to the corresponding restriction functor, the functor $\mu_!$ is induced by the composition $\mu : \mathcal{P}\circ \mathcal{P} \lrar \P$ regarded as a map in  $\BMod(\mathcal{P})$, and $(\Id+\Id)_!$ refers to the functor induced by the fold map $\mathcal{P}\bigsqcup\mathcal{P} \lrar \P$. To verify the commutativity, it suffices to observe that the associated square of right adjoints is commutative. 

\smallskip

Let us take a factorization $\mathcal{P}\circ \mathcal{P} \x{\varphi}{\lrar} \P^{\sB} \lrarsimeq \mathcal{P}$ in $\BMod(\mathcal{P})$ of $\mu$ into a cofibration followed by a weak equivalence. In particular, $\varphi$ exhibits $\P^{\sB}$ as a cofibrant model for $\P\in\BMod(\mathcal{P})^{*}$. Since $\P$ is intrinsically stable and cofibrant, it implies that the derived image of $\P^{\sB}\in\BMod(\mathcal{P})^{*}$ through the composed functor $(\Id+\Id)_!\circ\E$ is a model for  $\Sigma(\mathcal{P}\bigsqcup\mathcal{P})$, ie, the homotopy pushout of the fold map $\mathcal{P}\bigsqcup\mathcal{P} \lrar \P$ with itself. In other direction, the derived image of $\P^{\sB}$ under $\mu_!$ is the pushout $\P  \underset{\P\circ\P}{\sqcup}  \P^{\sB} =: B_\P$. Combined with the commutativity of \eqref{eq:goodness}, these observations prove that the derived image of $B_\P \in \BMod(\mathcal{P})_{\P/}$ through $\ovl{\E}$ is given by $\Sigma(\mathcal{P}\bigsqcup\mathcal{P}) \in \Op_C(\mathcal{S})_{\P/}$.

\smallskip

Let us now consider the following coCartesian square in $\BMod(\mathcal{P})$:
	\begin{equation}\label{eq:goodness1}
	\xymatrix{
		\P \sqcup (\P\circ\P)  \ar^{ \; \; \Id\sqcup\varphi}[r]\ar_{\Id+\mu}[d] & \P \sqcup \P^{\sB} \ar[d] \\
		\P \ar[r] & B_\P. \\
	}\end{equation}
Note that, when considered as a coCartesian square in $\BMod(\mathcal{P})_{\P/}$, this square is in fact homotopy coCartesian due to the fact that $\varphi : \P\circ\P \lrar \P^{\sB}$ is a cofibration between cofibrant $\P$-bimodules. Consider $\O$ as an object of $\BMod(\mathcal{P})_{\P/}$. The square \eqref{eq:goodness1} yields a fiber sequence:
$$  \Map^{\h}_{\BMod(\mathcal{P})_{\P/}}(B_\P, \O) \lrar \Map^{\h}_{\BMod(\mathcal{P})_{\P/}}(\P \sqcup \P^{\sB}, \O) \lrar \Map^{\h}_{\BMod(\mathcal{P})_{\P/}}(\P \sqcup (\P\circ\P), \O).$$
We shall complete the proof by showing each of these three spaces, in the order given, is naturally weakly equivalent to the respective space in the sequence \eqref{eq:goodnessmain}. First, argue that there exists a chain of weak equivalences:
$$  \Map^{\h}_{\BMod(\mathcal{P})_{\P/}}(B_\P, \O) \simeq \Map^{\h}_{\Op_C(\mathcal{S})_{\P/}}(\Sigma(\mathcal{P}\bigsqcup\mathcal{P}), \O) \simeq \Omega\Map^{\h}_{\Op_C(\mathcal{S})_{\P/}}(\mathcal{P}\bigsqcup\mathcal{P}, \O) .$$
Indeed, the first weak equivalence follows from the second paragraph, while the other is straightforward. Moreover, by adjunction we have a weak equivalence $$\Map^{\h}_{\Op_C(\mathcal{S})_{\P/}}(\mathcal{P}\bigsqcup\mathcal{P}, \O)\simeq \Map^{\h}_{\Op_C(\mathcal{S})}(\mathcal{P}, \O).$$ 
We thus obtain a weak equivalence $\Map^{\h}_{\BMod(\mathcal{P})_{\P/}}(B_\P, \O) \simeq \Omega\Map^{\h}_{\Op_C(\mathcal{S})}(\mathcal{P}, \O)$. By adjunction again, we obtain the following chains of weak equivalences:
$$ \Map^{\h}_{\BMod(\mathcal{P})_{\P/}}(\P \sqcup \P^{\sB}, \O) \simeq \Map^{\h}_{\BMod(\mathcal{P})}(\P^{\sB}, \O) \simeq \Map^{\h}_{\BMod(\mathcal{P})}(\P, \O), \; \text{and}$$
$$ \Map^{\h}_{\BMod(\mathcal{P})_{\P/}}(\P \sqcup (\P\circ\P), \O) \simeq \Map^{\h}_{\BMod(\mathcal{P})}(\P\circ\P, \O) \simeq \Map^{\h}_{\Coll_C(\cS)}(\I_C, \O).$$
Finally, notice that the latter space is weakly equivalent to $\underset{c\in C}{\prod}\Map^{\h}_{\cS}(1_{\cS}, \O(c;c))$.
\end{proof}

\begin{remark} According to the proof above, the map 
	$$ \Map^{\h}_{\BMod(\P)}(\P, \O) \lrar \underset{c\in C}{\prod}\Map^{\h}_{\cS}(1_{\cS}, \O(c;c)) $$ 
is weakly equivalent to the map $\Map^{\h}_{\BMod(\P)}(\P, \O) \lrar \Map^{\h}_{\BMod(\P)}(\P\circ\P, \O)$ that is induced by the composition $\mu : \P\circ\P \lrar \P$.
\end{remark}

The following explains the significance of \textit{adequate base categories}.
\begin{lemma}\label{l:goodbase} Suppose further that the model structure on $\cS$ is cofibrantly generated, that the monoidal unit $1_\cS$ is cofibrant, and that $\cS$ is adequate in the sense of Definition \ref{S8}. Then the statement of Theorem \ref{t:goodness} already holds true for every $\Sigma$-cofibrant operad in $\cS$.	
	\begin{proof} Let $\P \in \Op_C(\mathcal{S})$ be a $\Sigma$-cofibrant operad and let $f:\Q \lrarsimeq \P$ be a cofibrant resolution for $\P$. In particular, the operad $\Q$ is intrinsically stable by assumption. We deduce the lemma by observing that the adjunction $\adjunction*{}{\BMod(\Q)}{\BMod(\P)}{}$ induced by $f$ is a Quillen equivalence (see \cite[$\S$12.5]{Fresse1} and Proposition \ref{p:BP}).
	\end{proof}
\end{lemma}

We will now discuss the verification of the stability hypothesis in practice. It was proved by Dwyer-Hess that every \textbf{non-symmetric simplicial operad} is intrinsically stable in our sense (cf  \cite[Proposition 5.4]{Hess}). Expanding upon their result, we have obtained the following in \cite[$\S$5.1]{Hoang}. 
\begin{proposition}\label{p:goodbase} Suppose that $\cS$ is either $\Set_\Delta$ or $\sMod_{\textbf{k}}$ with $\textbf{k}$ being any commutative ring. Then $\cS$ is adequate, ie, every  cofibrant (symmetric) operad in $\cS$ is intrinsically stable.
\end{proposition}

Here is an application of the operadic Dold-Kan correspondence.
\begin{proposition}\label{p:goodbase1} Let $R$ be a commutative ring containing the field $\QQ$. Then the base category $\C_{\geqslant0}(R)$ is adequate.
		\begin{proof} Let $\O\in \Op_{C}(\C_{\geqslant0}(R))$ be a cofibrant operad. Consider the following square of left Quillen functors
		$$  \xymatrix{
			\BMod(\O)^* \ar[r]\ar_-{\E}[d] & \BMod(\rL_C\O)^* \ar^-{\E}[d] \\
			\Op_C(\C_{\geqslant0}(R))_{\O\bigsqcup\O/} \ar[r] & \Op_C(\sMod_{R})_{\rL_C(\O\bigsqcup\O)/} \\
		}	 $$
where the top horizontal functor is the pointed-analogue of the left Quillen functor $\BMod(\O) \lrarsimeq \BMod(\rL_C\O)$, while  the bottom horizontal functor is induced by $\rL_C : \Op_C(\C_{\geqslant0}(R)) \lrarsimeq \Op_C(\sMod_{R})$ (see Remark \ref{r:opbivar}). One can show that this square is commutative (again by verifying the commutativity of the associated square of right adjoints). We now claim that $\O$ is intrinsically stable due to the fact that $\rL_C\O$ is intrinsically stable by Proposition \ref{p:goodbase}, and that both the horizontal functors are a left Quillen equivalence identifying $\O$ to $\rL_C\O$.
	\end{proof}
\end{proposition}

Combining the above statements, we obtain the following.
\begin{corollary}\label{t:sum} Suppose that $\cS$ is any of the categories in $\{\Set_\Delta, \sMod_{\textbf{k}}, \C_{\geqslant0}(R)\}$ where $\textbf{k}$ and $R$ are commutative rings, with $R$ containing $\QQ$. Let $\P\in\Op_C(\mathcal{S})$ be a $\Sigma$-cofibrant operad and suppose given a map $\P\lrar\O$ between $C$-colored operads. Then there is a fiber sequence of spaces
	\begin{equation}\label{eq:goodnessmain1}
		\Omega\Map^{\h}_{\Op_C(\mathcal{S})}(\P, \O) \lrar \Map^{\h}_{\BMod(\P)}(\P, \O) \lrar \underset{c\in C}{\prod}\Map^{\h}_{\cS}(1_{\cS}, \O(c;c)).
	\end{equation}
\end{corollary}

\begin{example} In the case where $\cS=\Set_\Delta$, the above assertion extends a result of Ducoulombier \cite{Julien}, assuming that the operads are single-colored and that $\O(1) \simeq *$. More precisely, in this situation we obtain a weak equivalence
		\begin{equation}\label{eq:sum1}
 \Omega\Map^{\h}_{\Op_*(\Set_\Delta)}(\P, \O) \lrarsimeq \Map^{\h}_{\BMod(\P)}(\P, \O)
	\end{equation}
where $\Op_*(\Set_\Delta)$ refers to the category of  single-colored simplicial operads. 
\end{example}

\begin{remark} For the case $\cS=\C_{\geqslant0}(R)$ (or $\sMod_{\textbf{k}}$), an analogue of \eqref{eq:sum1} does not hold for nontrivial choices of the operad $\O$. Instead, in several cases of interest where $\O(1) \simeq R$, we have a fiber sequence of the form
	$$ \Omega\Map^{\h}_{\Op_*(\C_{\geqslant0}(R))}(\P, \O) \lrar \Map^{\h}_{\BMod(\P)}(\P, \O) \lrar R.$$
\end{remark}

\begin{example}\label{ex:end} Let $\cS$ and $\P$ be as in the statement of Corollary \ref{t:sum}.   For an ordered pair $(A,B)$ of objects in $\cS^{\times C}$, one can associate to it an \textbf{endomorphism $C$-collection}, denoted $\End_{A,B}$, and defined on each level by letting $$\End_{A,B}(c_1,\cdots,c_n;c) := \Map_{\cS}(A(c_1)\otimes\cdots\otimes A(c_n),B(c)).$$ 
In particular, the \textbf{endomorphism operad} associated to $A$ is given by $\End_{A}:=\End_{A,A}$. Let us assume that $A$ comes equipped with a $\P$-algebra structure classified by a map $\mu_A : \P\lrar\End_A$ of  operads. Then we have an adjunction
	$$ (-)\circ_{\P}A : \adjunction*{}{\BMod(\P)}{\Alg_{\P}(\cS)}{} : \End_{A,-} $$
where the left adjoint is given by the \textbf{relative composite product} over $\P$ (cf, eg, \cite{Fresse1, Rezk} for more details about these constructions).   Clearly the above adjunction becomes a Quillen adjunction as long as $A$ is termwise cofibrant. Now assume further that $A$ is \textit{bifibrant} (ie, both fibrant and cofibrant) as a $\P$-algebra, so that  
we have a weak equivalence $$\Map^{\h}_{\BMod(\P)}(\P, \End_A) \simeq \Map^{\h}_{\Alg_{\P}(\cS)}(A,A)$$
(cf \cite[Theorem 15.2.A]{Fresse1}), and so that the underlying space of $\Map_{\cS}(A(c),A(c))$ is a model for the corresponding derived mapping space. Applying Corollary \ref{t:sum} to the map $\mu_A$ hence gives us a fiber sequence:
\begin{equation}\label{eq:sum2}
\Omega\Map^{\h}_{\Op_C(\mathcal{S})}(\P, \End_A) \lrar \Map^{\h}_{\Alg_{\P}(\cS)}(A,A) \lrar \underset{c\in C}{\prod} \Map^{\h}_{\cS}(A(c),A(c)).
\end{equation}
\end{example}

\begin{example}\label{r:Rezkfiber} We will illustrate how the above example harmonizes naturally with a result obtained by Rezk \cite{Rezk}. For each model category $\textbf{M}$, we denote by $\textbf{M}_\infty^{\simeq}$ the \textbf{maximal $\infty$-groupoid} contained in the underlying $\infty$-category of $\textbf{M}$ (cf \cite{Luriehtt}). Moreover, for an object $X\in\textbf{M}$, we will write $$\Aut^{\h}_{\textbf{M}}(X) \subseteq\Map^{\h}_{\textbf{M}}(X,X)$$ for the subspace consisting of those that are  self-equivalences on $X$. It is known that $\Aut^{\h}_{\textbf{M}}(X)$ is a natural model for the loop space of $\textbf{M}_\infty^{\simeq}$ at $X$ (cf, eg, \cite{DKcor}). Now let $\cS$ and $\P$ be as in the statement of Corollary \ref{t:sum} again. For a termwise bifibrant object $A \in \cS^{\times C}$, according to Theorem 1.2.10 (and 1.2.15) of \cite{Rezk}, there is a fiber sequence
	\begin{equation}\label{eq:sum3}
		\Map^{\h}_{\Op_C(\mathcal{S})}(\P, \End_A) \lrar \Alg_{\P}(\cS)_\infty^{\simeq} \lrar \underset{C}{\prod}\cS_\infty^{\simeq}.
	\end{equation}
	Suppose further that $A$ is a bifibrant $\P$-algebra. Then there is a diagram of homotopy Cartesian squares of the form
	\begin{equation}\label{eq:Rezkfiber}
		\xymatrix{
			\Omega\Map^{\h}_{\Op_C(\mathcal{S})}(\P, \End_A) \ar[r]\ar[d] & \Aut^{\h}_{\Alg_{\P}(\cS)}(A) \ar[r]\ar[d] & \Map^{\h}_{\Alg_{\P}(\cS)}(A,A) \ar[d] \\
			\{\Id_A\} \ar[r] & \underset{c\in C}{\prod}\Aut^{\h}_{\cS}(A(c)) \ar[r] & \underset{c\in C}{\prod} \Map^{\h}_{\cS}(A(c),A(c)) \\
		} 
	\end{equation}
where the left square is induced by looping down the fiber sequence \eqref{eq:sum3} at $A$, and the right homotopy Cartesian square arises due to the fact that the forgetful functor $\Alg_{\P}(\cS) \lrar \cS^{\times C}$ creates weak equivalences.
\end{example}

\smallskip

\subsection{The second fiber sequence}\label{s:mappingen}

As the main interest, we now consider the mapping spaces between operads with non-fixed sets of colors. 

\smallskip

We let $\cS$ be a base category that is sufficient in the sense of Convention \ref{conv:adsu}, and let $\P\in\Op_{C}(\cS)$ be a $C$-colored operad in $\cS$ with $C$ being a fixed set of colors.

\begin{notation}\label{no:core} For an $\cS$-enriched category $\C$, we  denote by $$\cor(\C) := \Map^{\h}_{\Cat(\cS)}([0]_\cS, \C)$$ where $[0]_\cS$ denotes the category with a single object and with the mapping object being given by $1_\cS$, and refer to $\cor(\C)$ as the \textbf{core of $\C$}. 
\end{notation}

\begin{remark} We have that  $\pi_0(\cor(\C))$ represents the \textbf{equivalence classes} of objects in $\C$. Here by convention two objects in $\C$ are equivalent if they are isomorphic as objects in the homotopy category $\Ho(\C)$ (see $\S$\ref{s:22}). 
\end{remark}

\begin{example}\label{ex:coresim} When $\C$ is a fibrant simplicial category, we have a weak equivalence $\cor(\C) \simeq \sN(\C)^{\simeq}$, ie, the maximal $\infty$-groupoid contained in the homotopy coherent nerve of $\C$. To see this, it suffices to observe that 
	$$ \Map^{\h}_{\Cat(\Set_\Delta)}([0]_{\Set_\Delta}, \C) \simeq \Map_{\Cat_\infty}(*, \sN(\C)) $$
where the space on the right is the mapping space in the $\infty$-category of $\infty$-categories from the terminal $\infty$-category $*$ to $\sN(\C)$, which is exactly given by the maximal $\infty$-groupoid contained in $\Fun(*,\sN(\C)) \cong \sN(\C)$. 
\end{example}

\begin{example}\label{ex:coregen} For a commutative ring $\textbf{k}$, we have a Quillen adjunction $$\adjunction*{}{\Cat(\Set_\Delta)}{\Cat(\sMod_{\textbf{k}})}{}$$ induced by the free-forgetful adjunction $\adjunction*{}{\Set_\Delta}{\sMod_{\textbf{k}}}{}$. For a category $\C\in\Cat(\sMod_{\textbf{k}})$, we denote by $\C_\Delta \in \Cat(\Set_\Delta)$ its underlying simplicial category. By adjunction, we obtain a weak equivalence
	$$ \Map^{\h}_{\Cat(\Set_\Delta)}([0]_{\Set_\Delta}, \C_\Delta) \simeq \Map^{\h}_{\Cat(\sMod_{\textbf{k}})}([0]_{\sMod_{\textbf{k}}}, \C).$$
Combining this with the above example, we obtain a chain of weak equivalences $$ \cor(\C) \simeq \cor(\C_\Delta) \simeq \sN(\C_\Delta)^{\simeq}.$$
Moreover, this machinery also applies to the categories in $\C_{\geqslant0}(\textbf{k})$, as a consequence of the \textbf{categorical Dold-Kan correspondence} presented in \cite{Tabuada}.
\end{example}

Now suppose we are given an operad $\Q\in \Op_D(\cS)$ colored in some set $D$, and along with a map $f :  C \lrar D$ between the sets of colors. Recall by notations that $\Q_1 \in \Cat(\cS)$ signifies the underlying category of $\Q$ formed by the collection of unary operations, and $f^*\Q \in \Op_C(\cS)$ denotes the $C$-colored operad determined by the restriction of colors.

\begin{remark} It follows immediately from the definition that $$\Hom_{\Op(\cS)}(\P,\Q) \cong \underset{\alpha}{\coprod} \; \Hom_{\Op_C(\cS)}(\P,\alpha^*\Q)$$ where the coproduct ranges over the set of maps $ C \lrar D$. Nonetheless, when passing to the realm of homotopy theory, the circumstance might vary significantly, as we will see below.
\end{remark}

The following is the key step for simplifying the mapping spaces in $\Op(\cS)$.

\begin{theorem}\label{p:simplifying} There is a homotopy Cartesian square of spaces of the form
	\begin{equation}\label{eq:simplifying}
		\xymatrix{
			\Map^{\h}_{\Op_C(\cS)}(\P,f^*\Q)  \ar[r]\ar[d] & \Map^{\h}_{\Op(\cS)}(\P,\Q) \ar[d] \\
		\{*\} \ar[r]^{\underset{c\in C}{\prod}\{f(c)\}} &  \underset{C}{\prod}\cor(\Q_1) \;  \\
		} 
	\end{equation}
in which the bottom horizontal map is given by, for each $c\in C$, inserting the object $f(c) \in \Ob(\Q_1)$ into $\cor(\Q_1)$.
\begin{proof} Without loss of generality, we may assume that $\P\in\Op_C(\cS)$ is cofibrant. Since the initial $C$-colored operad $\I_C$ is cofibrant as an object of $\Op(\cS)$, it implies that $\P$ is as well cofibrant within $\Op(\cS)$ (cf \cite[Observation 4.1.1(ii)]{Hoang}). Now consider the following coCartesian square in $\Op(\cS)$:
	$$ \xymatrix{
		\I_C \coprod \I_C  \ar[r]\ar[d] & \I_C \coprod \P \ar[d] \\
		\I_C \ar[r] &  \P \;  \\
	} $$
Due to the discussion above, this square is automatically homotopy coCartesian when considered as a square in $\Op(\cS)_{\I_C/}$. We also regard $\Q$ as an object of $\Op(\cS)_{\I_C/}$ via the map $\eta_f :\I_C \lrar \Q$ that is determined uniquely by $f:C \lrar D$. Now, the above square yields a homotopy Cartesian square of spaces
	$$ \xymatrix{
		\Map^{\h}_{\Op(\cS)_{\I_C/}}(\P,\Q) \ar[r]\ar[d] & \Map^{\h}_{\Op(\cS)}(\P,\Q) \ar[d] \\
		\{\eta_f\} \ar[r] &  \Map^{\h}_{\Op(\cS)}(\I_C,\Q) \;  \\
	} $$
(In fact, a generalization of this holds true in any \textit{under $\infty$-category}, according to \cite[Lemma 5.5.5.12]{Luriehtt} with the opposite convention.) We shall complete the proof by showing the existence of natural weak equivalences:
  $$  \Map^{\h}_{\Op(\cS)_{\I_C/}}(\P,\Q) \simeq  \Map^{\h}_{\Op_C(\cS)}(\P,f^*\Q)  \;\; \text{and} \;\; \Map^{\h}_{\Op(\cS)}(\I_C,\Q) \simeq \underset{C}{\prod}\cor(\Q_1).$$
  
Let us start with the first weak equivalence. By Remark \ref{rem:canfib} and by the assumption that $\cS$ is sufficient (see Convention \ref{conv:adsu}), the adjunction \eqref{eq:signi} is a Quillen adjunction. Combining this  with the fact that $\I_C$ represents an initial object in $\Op_C(\cS)$, we obtain natural weak equivalences
$$ \Map^{\h}_{\Op_C(\cS)}(\P,f^*\Q) \simeq \Map^{\h}_{\Op_C(\cS)_{\I_C/}}(\P,f^*\Q) \simeq \Map^{\h}_{\Op(\cS)_{\I_C/}}(\P,\Q),$$
as expected.

\smallskip

For the second one, we will make use of the Quillen adjunction $\adjunction*{}{\Cat(\cS)}{\Op(\cS)}{}$ (see Remark \ref{r:signi1} and Remark \ref{rem:catop}). Combined with the observation that $\I_C$ can be identified with a category in $\cS$ exhibited as the $C$-fold coproduct of $[0]_\cS\in\Cat(\cS)$, that adjunction yields natural weak equivalences
$$ \Map^{\h}_{\Op(\cS)}(\I_C,\Q) \simeq \Map^{\h}_{\Cat(\cS)}(\I_C,\Q_1) \simeq \underset{C}{\prod}\Map^{\h}_{\Cat(\cS)}([0]_\cS,\Q_1) \, \x{\defi}{=} \, \underset{C}{\prod}\cor(\Q_1).$$
Moreover, under this identification, the map $\eta_f :\I_C \lrar \Q$ corresponds to exactly the embedding $\underset{c\in C}{\prod}\{f(c)\} \subseteq \underset{C}{\prod}\cor(\Q_1)$.
\end{proof}
\end{theorem}

Consequently, a connection between the mapping spaces in $\Op(\cS)$ and mapping spaces of operadic bimodules is achieved by combining the above theorem with Theorem \ref{t:goodness}.

\smallskip

To further elaborate, we shall now provide some additional examples derived from our current findings. 

\begin{example} First, consider the case where $\cS$ is the category of simplicial sets, $\P$ is $\Sigma$-cofibrant and the operad $\Q$ is such that $\Q(d;d') \simeq *$  for every $d,d'\in D$. Moreover, suppose we are given a map $f : \P \lrar \Q$ in $\Op(\Set_\Delta)$.  Clearly in this case we have $\cor(\Q_1) \simeq *$. Thus due to Theorem \ref{p:simplifying}, we get that
	$$ \Map^{\h}_{\Op(\cS)}(\P,\Q) \simeq \Map^{\h}_{\Op_C(\cS)}(\P,f^*\Q).$$ 
Furthermore, combining this with Theorem \ref{t:goodness}, we obtain that
  $$ \Omega\Map^{\h}_{\Op(\cS)}(\P,\Q) \simeq \Omega\Map^{\h}_{\Op_C(\cS)}(\P,f^*\Q) \simeq \Map^{\h}_{\BMod(\P)}(\P, f^*\Q)$$
 where the first loop space is taken at the base point $f : \P \lrar \Q$, while the second is taken at the map $\P\lrar f^*\Q$ induced by $f$.
\end{example}

\begin{example} Now consider the case where $\cS = \C_{\geqslant0}(R)$ with $R$ being a commutative ring containing $\QQ$, while $\P$ is $\Sigma$-cofibrant, and $\Q$ satisfies that $\Q(d;d') \simeq R$ for every $d,d'\in D$. As above, suppose we are given a map $f : \P \lrar \Q$ in $\Op(\C_{\geqslant0}(R))$. According to Example \ref{ex:coregen}, we may show that $\cor(\Q_1)$ is weakly equivalent to the \textit{Eilenberg–MacLane space} $K(R^\times, 1)$ in which as usual $R^\times$ signifies the group of units. Thus, in this case we obtain two fiber sequences
	$$ \Omega\Map^{\h}_{\Op_C(\C_{\geqslant0}(R))}(\P,f^*\Q) \lrar \Map^{\h}_{\BMod(\P)}(\P, f^*\Q) \lrar \underset{C}{\prod} \, R \,, \; \text{and}$$  
	$$ \Map^{\h}_{\Op_C(\C_{\geqslant0}(R))}(\P,f^*\Q) \lrar \Map^{\h}_{\Op(\C_{\geqslant0}(R))}(\P,\Q) \lrar \underset{C}{\prod} \, K(R^\times, 1) $$
respectively according to Theorems \ref{t:goodness} and \ref{p:simplifying}. 
\end{example}

\begin{example}\label{ex:ends} We now revisit our examination of the general base category $\cS$. Let us consider the case where $\Q = \End_\cS$, ie, the operad in $\cS$ whose colors are chosen from the collection of objects in $\cS$, and such that the spaces of operations are given by
	$$  \End_\cS(X_1,\cdots,X_n;X)  := \Map_{\cS}(X_1\otimes\cdots\otimes X_n,X).$$ 
Suppose we are given a collection $A = \{A(c)\}_{c\in C}$ of objects in $\cS$, regarded as a map from $C$ to $\Ob(\cS)$. Then the restriction along $A$ yields a $C$-colored operad $A^*(\End_\cS)$, which exactly coincides with the endomorphism operad $\End_{A}$ (see Example \ref{ex:end}). On other hand, notice that the underlying category $(\End_\cS)_1$ is nothing but $\cS \in \Cat(\cS)$. Theorem \ref{p:simplifying} hence yields a fiber sequence of the form
\begin{equation}\label{eq:ends}
	\Map^{\h}_{\Op_C(\cS)}(\P,\End_{A}) \lrar \Map^{\h}_{\Op(\cS)}(\P,\End_\cS) \lrar \underset{C}{\prod}\cor(\cS).
\end{equation}
\end{example}

\begin{remark}\label{r:ends1} One of the main interests in the operad $\End_\cS$ is that each map $\varphi : \P \lrar \End_\cS$ in $\Op(\cS)$ can be identified with a $\P$-algebra in $\cS$, whose underlying objects are given by the collection $\{\varphi(c)\}_{c\in C} \in \cS^{\times C}$. One may further expect that an analogue remains true when moving into the realm of homotopy theory, ie, there might be a weak equivalence of the form
	\begin{equation}\label{eq:ends1}
		\Map^{\h}_{\Op(\cS)}(\P,\End_\cS) \simeq \Alg_{\P}(\cS)_\infty^\simeq.
	\end{equation}
For example, this is the case for every simplicial operad (or $\infty$-\textit{operad}), according to  \cite[$\S$2.1.4]{Lurieha}. Furthermore, an advantage is that when \eqref{eq:ends1} holds true in $\cS$, then combined with \eqref{eq:ends}, it yields a homotopy Cartesian square  
$$ \xymatrix{
	\Map^{\h}_{\Op_C(\cS)}(\P,\End_{A})  \ar[r]\ar[d] & \Alg_{\P}(\cS)_\infty^\simeq \ar[d] \\
	\{A\} \ar[r] &  \underset{C}{\prod}\cor(\cS), \;  \\
} $$
which will provide a generalization for \cite[Theorem 1.2.10]{Rezk} of Rezk (see also Example \ref{r:Rezkfiber}). 
\end{remark}

\subsection{Double loop spaces of mapping spaces between enriched operads}\label{s:doule}

In this last subsection, we shall provide a proof of the statement \ref{t:double1}.

\begin{conv} We say that the base category $\cS$ is \textbf{abundant} if the following conditions hold:
	
		\smallskip
	
	(i) $\cS$ is sufficient in the sense of Convention \ref{conv:adsu}.
	
	\smallskip
	
	(ii) $\cS$ satisfies Lurie's \textbf{invertibility hypothesis} \cite[Definition A.3.2.12]{Luriehtt}.
	
		\smallskip
	
   (iii)	The monoidal unit $1_\cS$ is cofibrant.
\end{conv}

\begin{remark} Roughly speaking, the invertibility hypothesis requires that, for any category $\C\in\Cat(\cS)$ containing a morphism $f$, localizing $\C$ at $f$ does not change the homotopy type of $\C$ as long as $f$ represents an isomorphism in $\Ho(\C)$. This is quite common in practice, as discussed in \cite[$\S$3]{Hoang}. In particular, all the base categories of Example \ref{ex:basis} satisfy the invertibility hypothesis, and furthermore, are abundant in the sense above.
\end{remark}

We will let $\P\in \Op_C(\cS)$ be a $C$-colored operad in $\cS$, where 
	$C$ is a fixed set of colors. Moreover, suppose we are given a map $f : \P \lrar \O$ in $\Op(\cS)$. As usual, $f$ induces a map $\P\lrar f^*\O$ in $\Op_C(\cS)$. We also regard the latter as a map in $\BMod(\P)$.

\begin{theorem}\label{t:double} Suppose that the base category $\cS$ is abundant, and suppose that $\P\in \Op_C(\cS)$ is cofibrant and intrinsically stable (cf Definition \ref{S8}). Then there is a natural weak equivalence 
	$$  \Omega^2\Map^{\h}_{\Op(\cS)}(\P,\O) \simeq \Omega\Map^{\h}_{\BMod(\P)}(\P,f^*\O) $$
	in which the double loop space on the left is taken at the base point $f$, and while the loop space on the right is taken at the map $\P\lrar f^*\O$.
\end{theorem}

The proof will leverage an approach analogous to the one we used in proving \cite[Proposition 5.2.5]{Hoang}. In what follows, we assume that the hypothesis of the theorem above holds.

\begin{notation} We let $[1]_\cS \in \Cat(\cS)$ denote the category with objects $0, 1$ and with the mapping spaces $\Map_{[1]_\cS}(0,1) = \Map_{[1]_\cS}(0,0) = \Map_{[1]_\cS}(1,1) = 1_\cS$ and  $\Map_{[1]_\cS}(1,0) = \emptyset_\cS$. Next, denote by $[1]_\cS^{\sim}$ the category which is the same as the previous one except that $\Map_{[1]^{\sim}_\cS}(1,0) = 1_\cS$.
\end{notation}

\begin{remark}\label{r:01} For a category $\C \in \Cat(\cS)$, while a map $[0]_\cS \lrar \C$ corresponds to a single object in $\C$, a map of the form $[1]_\cS \lrar \C$ is classified by a single morphism in $\C$. Namely, the latter consists of a pair $(x,y)$ of objects in $\C$ and a map $1_\cS \lrar \Map_\C(x,y)$. In light of this, along with the fact that a trivial fibration in $\Cat(\cS)$ is in particular a levelwise trivial fibration (see $\S$\ref{s:22}), and that $1_\cS$ is cofibrant, one can show that the canonical map $[0]_\cS\coprod[0]_\cS \x{(0,1)}{\lrar} [1]_\cS$ is a cofibration.           
\end{remark}

Next, we take a factorization $$[1]_\cS \lrar \mathcal{E} \lrarsimeq [1]_\cS^{\sim}$$ of the canonical map $[1]_\cS \lrar [1]_\cS^{\sim}$ into a cofibration followed by a trivial fibration. We now obtain a sequence of maps
 \begin{equation}\label{eq:otimes}
[0]_\cS\coprod[0]_\cS \lrar [1]_\cS \lrar \mathcal{E} \lrarsimeq [1]_\cS^{\sim} \lrarsimeq [0]_\cS
\end{equation}
such that the first two maps are cofibrations, while the others are weak equivalences. 

\smallskip

We will write $-\otimes-$ for the  \textit{tensor product of categories in} $\cS$. For more details, the tensor product $\C \otimes \D$ of two categories $\C$ and $\D$ has as set of objects $\Ob(\C \otimes \D) := \Ob(\C) \times \Ob(\D)$ and such that for every $c,c'\in \Ob(\C)$ and $d,d'\in\Ob(\D)$ we have
  $$ \Map_{\C \otimes \D}((c,d),(c',d')) := \Map_\C(c,c') \otimes \Map_\D(d,d').$$

 Now applying $\I_C\otimes(-)$ (where $\I_C$ is viewed as an $\cS$-enriched category) to the sequence \eqref{eq:otimes}, we obtain another sequence of maps in $\Cat(\mathcal{S})$:
 \begin{equation}\label{eq:icotimes}
 	\I_C \coprod \I_C \lrar \I_C \otimes [1]_\cS \lrar \I_C \otimes \mathcal{E} \lrarsimeq  \I_C \otimes [1]_\cS^{\sim} \lrarsimeq \I_C.
 \end{equation}

\begin{remark}\label{r:01i} Note that, for any category $\C \in \Cat(\cS)$, there is an identification $\I_C \otimes \C \cong \underset{C}{\coprod} \C$. Due to this, we get that the first two maps in the sequence above are cofibrations. Moreover, we may obtain that the last two maps are again weak equivalences, after verifying  that the class of Dwyer-Kan equivalences between levelwise cofibrant categories is stable under (arbitrary) coproducts.               
\end{remark}

We are concerned with several $\cS$-enriched operads, including $\A$, $\B$, $\ovl{\A}$, and $\ovl{\B}$, which are organized into the following diagram of coCartesian squares in $\Op(\cS)$:

\begin{equation}\label{eq:pushoutsss}
	\xymatrix{
		\I_C \ar[r]\ar_{i_0}[d] & \P \ar[d] \\
		\I_C \, \coprod \, \I_C   \ar[r]\ar[d] & \P \, \coprod \,  \I_C  \ar[r]\ar[d] & \P \, \coprod \, \P \ar[d] \\
		\I_C \otimes [1]_\cS   \ar[r]\ar_{\psi}[d] & \ovl{\A}  \ar[r]\ar^{\ovl{\varphi}}[d] & \A \ar^{\varphi}[d] \\
		\I_C \otimes \mathcal{E}  \ar[r]\ar_{\simeq}[d] & \ovl{\B}  \ar[r]\ar^{\simeq}[d] & \B \ar^{\simeq}[d] \\
		\I_C \ar[r] & \P \ar[r] & \P \, \underset{\I_C}{\coprod} \, \P . \\
	}
\end{equation}

\begin{remark}\label{r:pushoutsss} By Remark \ref{r:01i} and the cofibrant requirement on $\P$, all maps in this diagram are cofibrations; except for the three bottom vertical maps, which are all weak equivalences. (Among these, the last two are a homotopy cobase change of the weak equivalence $\I_C \otimes \mathcal{E} \lrarsimeq \I_C$.) Moreover, the middle column is nothing but the (derived) image of the first column through the left Quillen functor $\eta_! : \Op(\cS)_{\I_C//\I_C} \lrar \Op(\cS)_{\P//\P}$ induced by the unit map $\eta : \I_C \lrar \P$. We also regard the three squares on the right as coCartesian squares in $\Op(\cS)_{\P//\P}$. Accordingly, we may consider $\A$, $\B$, $\ovl{\A}$, and $\ovl{\B}$ as objects in $\Op(\cS)_{\P//\P}$.    
\end{remark}

We regard $\Op(\cS)_{\P//\P}$ as a pointed model category with $\P$ as the base point, so that we may define the suspension functor $\Sigma(-)$ on this category.
\begin{lemma}\label{l:singlesus} The induced map $\Sigma(\varphi) : \Sigma(\A) \lrar \Sigma(\B)$ is a weak equivalence in $\Op(\cS)_{\P//\P}$.
\end{lemma}
\begin{proof} In fact, this claim can be derived from the proof of \cite[Lemma 5.2.10]{Hoang}. For readers' convenience, we still present here a detailed proof. First, by Remark \ref{r:pushoutsss} the map $\varphi$ is a homotopy cobase change in $\Op(\cS)_{\P//\P}$ of the map  $\ovl{\varphi} : \ovl{\A} \lrar \ovl{\B}$. Thus we just need to prove that the map $\Sigma(\ovl{\varphi})$ is a weak equivalence. Moreover, note that $\ovl{\varphi}$ coincides with the (derived) image of the map $\psi : \I_C\otimes [1]_\cS   \lrar \I_C\otimes \mathcal{E}$ through the left Quillen functor $\Op(\cS)_{\I_C//\I_C} \lrar \Op(\cS)_{\P//\P}$ (by Remark \ref{r:pushoutsss} again). Therefore, it will suffice to verify that $\Sigma(\psi)$ is a weak equivalence (in $\Cat(\cS)_{\I_C//\I_C}$). By the weak equivalence $\I_C\otimes \mathcal{E} \lrarsimeq \I_C$, we just need to show that the fold map 
$$\I_C\otimes \mathcal{E} \underset{\I_C\otimes [1]_\cS}{\coprod}\I_C\otimes \mathcal{E} \lrar \I_C\otimes \mathcal{E}$$
is a weak equivalence. For this, it suffices to prove that the fold map $$\mathcal{E} \underset{[1]_\cS}{\coprod}\mathcal{E} \lrar \mathcal{E}$$ is a weak equivalence (see Remark \ref{r:01i}), which immediately follows from the invertibility hypothesis on $\cS$.
\end{proof}

\begin{remark}\label{r:varphi} Combining the above lemma with Remark \ref{r:pushoutsss}, we now obtain a sequence of maps $\A \lrar \B \lrarsimeq \P\underset{\I_C}{\coprod}\P$, each of which becomes a weak equivalence after taking suspension:
	$$  \Sigma(\A) \lrarsimeq \Sigma(\B) \lrarsimeq \Sigma(\P\underset{\I_C}{\coprod}\P).$$
In light of this, we will regard $\Sigma(\A)$ and $\Sigma(\B)$ as two convenient models for the suspension $\Sigma(\P\underset{\I_C}{\coprod}\P)$.
\end{remark}

\begin{remark}\label{r:rho} We will be also interested in the following objects in $\Op(\cS)_{\P//\P}$:
	$$ \U := \P \underset{\I_C\coprod\I_C}{\coprod}\I_C \otimes [1]_\cS \; \;\; \text{and} \;\;\; \V := \P \underset{\I_C\coprod\I_C}{\coprod}\I_C \otimes \mathcal{E} \;.$$
According to the diagram \eqref{eq:pushoutsss}, these can be put into the following diagram of (homotopy) coCartesian squares in $\Op(\cS)_{\P//\P}$:
\begin{equation}
	\xymatrix{
	\P \, \coprod \, \P  \ar[r]^{\;\;\;\; \Id + \Id}\ar[d] & \P \ar[d] \\
	\A \ar[r]\ar_\varphi[d] &  \U \ar^\rho[d]  \\
	\B \ar[r] &  \V \\
} 
\end{equation}
where $\rho : \U \lrar \V$ is the canonical map induced by the map $[1]_\cS \lrar \mathcal{E}$. Due to this, we obtain that the induced map $$\Sigma(\rho) : \Sigma(\U) \lrarsimeq \Sigma(\V)$$ is a weak equivalence as well. 
\end{remark}

We are now in position to prove the main theorem of this subsection.

\smallskip

\begin{proof}[\underline{Proof of Theorem~\ref{t:double}}] The idea is first to represent the  spaces $\Omega^2\Map^{\h}_{\Op(\cS)}(\P,\O)$ and $\Omega\Map^{\h}_{\BMod(\P)}(\P,f^*\O)$ as the homotopy fibers of certain maps between mapping spaces in $\Op(\mathcal{S})_{\P/}$. 
	
\smallskip	
	
\textbf{Step 1.}	By adjunctions, we have a chain of weak equivalences $$\Map^{\h}_{\Op_C(\mathcal{S})}(\P, f^*\O) \, \simeq \, \Map^{\h}_{\Op_C(\mathcal{S})_{\P/}}(\P\bigsqcup\P, f^*\O) \, \simeq \, \Map^{\h}_{\Op(\mathcal{S})_{\P/}}(\P\underset{\I_C}{\coprod}\P, \O).$$ 
	 More concretely, the second weak equivalence follows from the canonical Quillen adjunction $\adjunction*{}{\Op_C(\mathcal{S})_{\P/}}{\Op(\mathcal{S})_{\P/}}{}$ (see the proof of Theorem \ref{p:simplifying}), along with the fact that the coproduct $\P\bigsqcup\P \in \Op_C(\mathcal{S})$, when considered as an object of $\Op(\cS)$, agrees with the pushout $\P\underset{\I_C}{\coprod}\P \in \Op(\mathcal{S})$. In the same fashion, we may obtain that
	 \begin{gather*}
	 	\underset{c\in C}{\prod}\Map^{\h}_{\cS}(1_{\cS}, f^*\O(c;c))  \simeq  \Map^{\h}_{\Op_C(\mathcal{S})}(\Free(\I_C), f^*\O)  \\
	 	\simeq \Map^{\h}_{\Op_C(\mathcal{S})_{\P/}}(\P\bigsqcup\Free(\I_C), f^*\O) \simeq \Map^{\h}_{\Op(\mathcal{S})_{\P/}}(\P\underset{\I_C}{\coprod}\Free(\I_C), \O)
	 \end{gather*}	 
	in which $\Free(\I_C)$ signifies the free $C$-colored operad generated by the $C$-collection $\I_C$. Thus we may rewrite the fiber sequence \eqref{eq:goodnessmain} for the map $\P\lrar f^*\O$ as
	$$  \Omega\Map^{\h}_{\Op(\mathcal{S})_{\P/}}(\P\underset{\I_C}{\coprod}\P, \O) \lrar \Map^{\h}_{\BMod(\P)}(\P,f^*\O) \lrar \Map^{\h}_{\Op(\mathcal{S})_{\P/}}(\P\underset{\I_C}{\coprod}\Free(\I_C), \O).$$
	On other hand, the fiber sequence \eqref{eq:simplifying} can equivalently be rewritten as
	$$ 	\Map^{\h}_{\Op(\mathcal{S})_{\P/}}(\P\underset{\I_C}{\coprod}\P, \O) \lrar \Map^{\h}_{\Op(\cS)}(\P,\O) \lrar \Map^{\h}_{\Op(\mathcal{S})_{\P/}}(\P\coprod\I_C, \O).$$
	Here, we made use of the identifications $$\Map^{\h}_{\Op(\mathcal{S})_{\P/}}(\P\coprod\I_C, \O) \simeq \Map^{\h}_{\Op(\mathcal{S})}(\I_C, \O) \simeq \underset{C}{\prod}\cor(\O_1)$$ (see the proof of Theorem \ref{p:simplifying} again). Consequently, we obtain two fiber sequences as follows
	$$
	\Omega\Map^{\h}_{\BMod(\P)}(\P,f^*\O) \lrar	 \Omega\Map^{\h}_{\Op(\mathcal{S})_{\P/}}(\P\underset{\I_C}{\coprod}\Free(\I_C), \O) \x{\alpha}{\lrar} \Omega\Map^{\h}_{\Op(\mathcal{S})_{\P/}}(\P\underset{\I_C}{\coprod}\P, \O),$$
	$$  \Omega^2\Map^{\h}_{\Op(\cS)}(\P,\O) \lrar \Omega^2\Map^{\h}_{\Op(\mathcal{S})_{\P/}}(\P\coprod\I_C, \O) \x{\beta}{\lrar} \Omega\Map^{\h}_{\Op(\mathcal{S})_{\P/}}(\P\underset{\I_C}{\coprod}\P, \O).$$

	\smallskip

\textbf{Step 2.} To complete the proof, we will find convenient models for the two maps $\alpha$ and $\beta$, and further describe how these can be related via a pair of natural weak equivalences.	

\smallskip

 Notice first that, when considered as an object of $\Op(\mathcal{S})$, the operad $\Free(\I_C)$ can be modeled by 
	 $$\Free(\I_C) \, \cong \, \I_C \, \underset{\I_C \coprod \I_C}{\coprod} \, \I_C\otimes [1]_\cS \in \Op(\mathcal{S}).$$
Due to this, we obtain an identification
	\begin{equation}\label{eq:U}
		\P\underset{\I_C}{\coprod}\Free(\I_C) \cong \P \underset{\I_C\coprod\I_C}{\coprod}\I_C \otimes [1]_\cS \; \x{\defi}{=} \; \U.
	\end{equation}
We hence obtain a model for the domain of $\alpha$:
$$ \Omega\Map^{\h}_{\Op(\mathcal{S})_{\P/}}(\P\underset{\I_C}{\coprod}\Free(\I_C), \O)  \simeq \Map^{\h}_{\Op(\mathcal{S})_{\P/}}(\Sigma(\U), \O).$$
On other hand, by Remark \ref{r:varphi} the codomain of $\alpha$ can naturally be modeled as:
$$ \Omega\Map^{\h}_{\Op(\mathcal{S})_{\P/}}(\P\underset{\I_C}{\coprod}\P, \O) \simeq \Map^{\h}_{\Op(\mathcal{S})_{\P/}}(\Sigma(\A), \O).$$
Moreover, under these identifications, the map $\alpha$ is  weakly equivalent to the map $$\Map^{\h}_{\Op(\mathcal{S})_{\P/}}(\Sigma(\U), \O) \x{\widetilde{\alpha}}{\lrar} \Map^{\h}_{\Op(\mathcal{S})_{\P/}}(\Sigma(\A), \O)$$ induced by the canonical map $\A \lrar \U$ (see Remark \ref{r:rho}).

\smallskip

	Next, because of the fact that $\P\coprod\I_C$ agrees with the (derived) image of $\I_C \coprod \I_C$ through the functor $\eta_! : \Op(\cS)_{\I_C//\I_C} \lrar \Op(\cS)_{\P//\P}$ (see Remark \ref{r:pushoutsss}), we obtain that 
		\begin{equation}\label{eq:V}
		\Sigma(\P\coprod\I_C) \simeq \eta_!\Sigma(\I_C \coprod \I_C) \simeq \eta_!(\I_C\underset{\I_C\coprod\I_C}{\coprod}\I_C \otimes \mathcal{E}) \cong \P \underset{\I_C\coprod\I_C}{\coprod}\I_C \otimes \mathcal{E} \; \x{\defi}{=} \; \V.
	\end{equation}	
In light of this, the domain of $\beta$ can be modeled as:
$$  \Omega^2\Map^{\h}_{\Op(\mathcal{S})_{\P/}}(\P\coprod\I_C, \O) \simeq \Map^{\h}_{\Op(\mathcal{S})_{\P/}}(\Sigma(\V), \O).$$
On the other hand, the codomain of $\beta$ can  be represented as:
$$ \Omega\Map^{\h}_{\Op(\mathcal{S})_{\P/}}(\P\underset{\I_C}{\coprod}\P, \O) \simeq \Map^{\h}_{\Op(\mathcal{S})_{\P/}}(\Sigma(\B), \O)$$
(see Remark \ref{r:varphi} again). Furthermore, under the above identifications, the map $\beta$ is by construction weakly equivalent to the map $$\Map^{\h}_{\Op(\mathcal{S})_{\P/}}(\Sigma(\V), \O) \x{\widetilde{\beta}}{\lrar} \Map^{\h}_{\Op(\mathcal{S})_{\P/}}(\Sigma(\B), \O)$$ that is induced by the canonical map $\B \lrar \V$.

\smallskip

We now obtain a commutative square 
$$ 	\xymatrix{
	\Map^{\h}_{\Op(\mathcal{S})_{\P/}}(\Sigma(\V), \O)  \ar^{\widetilde{\beta}}[r]\ar[d]_\simeq & \Map^{\h}_{\Op(\mathcal{S})_{\P/}}(\Sigma(\B), \O) \ar[d]^\simeq \\
	\Map^{\h}_{\Op(\mathcal{S})_{\P/}}(\Sigma(\U), \O) \ar_{\widetilde{\alpha}}[r] &  \Map^{\h}_{\Op(\mathcal{S})_{\P/}}(\Sigma(\A), \O) \;  \\
} $$
such that the top and bottom horizontal maps are respectively weakly equivalent to $\beta$ and $\alpha$, and while the two vertical maps are all weak equivalences (see Lemma \ref{l:singlesus} and Remark \ref{r:rho}). Thus we get a natural weak equivalence between homotopy fibers $$   \Omega^2\Map^{\h}_{\Op(\cS)}(\P,\O) \lrarsimeq \Omega\Map^{\h}_{\BMod(\P)}(\P,f^*\O)$$
as expected.
\end{proof}

\begin{remark} As in Lemma \ref{l:goodbase}, if $\cS$ is already an adequate base category then the statement of Theorem \ref{t:double} holds for every $\Sigma$-cofibrant operad in $\cS$. In particular, this is the case when $\cS$ is any of the categories in $\{\Set_\Delta, \sMod_{\textbf{k}}, \C_{\geqslant0}(R)\}$ with $\textbf{k}$ being any commutative ring and $R$ being a commutative ring containing $\QQ$.
\end{remark}

\begin{example} We may use Theorem \ref{t:double} to investigate mapping spaces between $\infty$-categories (or alternatively, simplicial categories). Suppose we are given a functor $ f : \C \lrar \D$ between $\infty$-categories. We will write $\textbf{S} := (\Set_\Delta)_\infty$ to denote the $\infty$-category of spaces. We may view the category $\Fun(\C^{\op} \times \C , \textbf{S})$ as a model for the category of $\C$-bimodules. Then we may regard $\C$ as a bimodule over itself via the functor
	$$ \Map_\C : \C^{\op} \times \C \lrar \textbf{S} , \; (x,y) \mapsto \Map_\C(x,y).$$
Also, the restriction along $f$ exhibits $\D$ as a $\C$-bimodule via the functor
$$ \Map_{f^*\D} : \C^{\op} \times \C \lrar \textbf{S}  , \; (x,y) \mapsto \Map_\D(f(x),f(y)).$$
Accordingly, we may rewrite
  $$ \Map^{\h}_{\BMod(\C)}(\C,f^*\D) \simeq \Map_{\Fun(\C^{\op} \times \C , \textbf{S})}(\Map_\C, \Map_{f^*\D}).$$
Next, after passing to the framework of simplicial categories, the unstraightening of $\Map_\C$ is exactly a model for the \textbf{twisted arrow} $\infty$-\textbf{category} $\Tw(\C)$, which comes together with a \textit{left fibration} $\Tw(\C) \lrar \C^{\op} \times \C$ (cf \cite[Proposition 5.2.1.11]{Lurieha}). Let us now write $\D^\star$ for the composed functor  
$$ \Tw(\C) \lrar \C^{\op} \times \C \x{\Map_{f^*\D}}{\lrar} \textbf{S},$$
which sends an edge $x \lrar y$ in $\C$ (viewed as an object of $\Tw(\C)$) to $\Map_\D(f(x),f(y))$. Due to the straightening-unstraightening adjunction, we may show that there is a weak equivalence of spaces
 $$\lim\D^\star  \simeq  \Map_{\Fun(\C^{\op} \times \C , \textbf{S})}(\Map_\C, \Map_{f^*\D}).$$
Combining this with Theorem \ref{t:double}, we now obtain a weak equivalence
  $$  \Omega^2\Map_{\Cat_\infty}(\C,\D)  \simeq \Omega\lim\D^\star.$$
\end{example}

\bigskip

\section*{Acknowledgements}

The author is grateful to the handling editor for the care and dedication devoted to the review process, and to the referees for their valuable suggestions, which have improved the presentation of the manuscript.

\bigskip

\bibliographystyle{amsplain}

\end{document}